\documentclass[reqno]{amsart}
\usepackage{url}
\usepackage{amssymb}
\usepackage{graphicx}
\usepackage[colorinlistoftodos]{todonotes}
\usepackage{amsmath}
\usepackage{amsthm}

\newtheorem{theorem}{Theorem}[section]
\newtheorem{lemma}[theorem]{Lemma}

\theoremstyle{definition}
\newtheorem{definition}[theorem]{Definition}

\theoremstyle{remark}

\DeclareMathOperator{\pend}{pend}
\DeclareMathOperator{\pond}{pond}

\DeclareMathOperator{\ped}{ped}
\DeclareMathOperator{\mypod}{pod}

\numberwithin{equation}{section}
\usepackage{algorithmic}
\usepackage[top=3cm,bottom=2cm,right=2cm,left=2cm]{geometry}
\makeatletter
\tikzset{ 
reuse path/.code={\pgfsyssoftpath@setcurrentpath{#1}} 
} 
\tikzset{even odd clip/.code={\pgfseteorule}, 
protect/.code={ 
\clip[overlay,even odd clip,reuse path=#1] 
(current bounding box.south west) rectangle (current bounding box.north east)
; 
}} 
\makeatother 
\usetikzlibrary{3d,arrows.meta,decorations.markings,perspective}
\tikzset{->-/.style={decoration={
  markings,
  mark=at position #1 with {\arrow{>}}},postaction={decorate}},
  ->-/.default=0.55}
\pgfmathsetmacro{\myaz}{15}

\begin{document}

\title[PEND and POND Partitions]{Explaining Unforeseen Congruence Relationships Between PEND and POND Partitions via an Atkin--Lehner Involution}

\author[J. A. Sellers]{James A. Sellers}
\address[J. A. Sellers]{Department of Mathematics and Statistics, University of Minnesota Duluth, Duluth, MN 55812, USA}
\email{jsellers@d.umn.edu}

\author[N. A. Smoot]{Nicolas Allen Smoot}
\address[N. A. Smoot]{Faculty of Mathematics, University of Vienna, Oskar-Morgenstern-Platz 1, 1090 Vienna, Austria}
\email{nicolas.allen.smoot@univie.ac.at }
	
	

\begin{abstract}
For the past several years, numerous authors have studied POD and PED partitions from a variety of perspectives.  These are integer partitions wherein the odd parts must be distinct (in the case of POD partitions) or the even parts must be distinct (in the case of PED partitions). 

More recently, Ballantine and Welch were led to consider POND and PEND partitions, which are integer partitions wherein the odd parts {\bf cannot} be distinct (in the case of POND partitions) or the even parts {\bf cannot} be distinct (in the case of PEND partitions).  
Soon after, the first author proved the following results via elementary $q$-series identities and generating function manipulations, along with mathematical induction: 
For all  $\alpha \geq 1$ and all $n\geq 0,$  
\begin{align*}
\pend\left(3^{2\alpha +1}n+\frac{17\cdot 3^{2\alpha}-1}{8}\right) 
&\equiv 0 \pmod{3}, \textrm{\ \ \ and} \\
\pond\left(3^{2\alpha +1}n+\frac{23\cdot 3^{2\alpha}+1}{8}\right) 
&\equiv 0 \pmod{3}
\end{align*}
where $\pend(n)$ counts the number of PEND partitions of weight $n$ and $\pond(n)$ counts the number of POND partitions of weight $n$.

In this work, we revisit these families of congruences, and we show a relationship between them via an Atkin--Lehner involution.  From this relationship, we can show that, once one of the above families of congruences is known, the other follows immediately.  
\end{abstract}

\maketitle

\section{Introduction}

The theory of partition congruences has a long history, stretching back to Ramanujan's groundbreaking work more than a century ago \cite{Ramanujan}.  It is well-known that  congruence relationships between the coefficients of modular forms vary substantially in their difficulty: some can be proved quite easily by techniques used by Ramanujan himself, while others continue to resist proofs to this day.

One important recent development in this subject is the study of \textit{multiplicities} between modular congruences.  This is the phenomenon in which a congruence property for a given linear progression in the coefficients of one modular form can manifest in another linear progression for the coefficients of a different modular form.  Recent work in this subject include \cite{Chern}, \cite{GarvanM} (also referenced in \cite[Section 4]{Garvan0}), and \cite{Garvan0}.

Here we report on the multiplicity between two families of congruences, modulo 3, exhibited by the PEND and POND partition functions that were originally studied by Ballantine and Welch \cite{BalWel}.  What is especially interesting about this example is that the associated progressions corresponding to the POND partition function also exhibit a parity condition which the associated progressions for the PEND function lack.  We show how the mapping used to demonstrate the multiplicity between these two sets of congruences must take this difference into account.

\subsection{PEND and POND Partitions}

A partition $\lambda$ of a positive integer $n$ is a sequence $(\lambda_1, \lambda_2, \dots, \lambda_r)$ such that $\lambda_1 \geq  \lambda_2\geq  \dots\geq  \lambda_r \geq 1$ and $\lambda_1+\lambda_2+\dots+\lambda_r = n$.  
Partitions wherein the parts are distinct have long played a key role in the theory of partitions, dating back to Euler's discovery and proof that the number of partitions of weight $n$ into distinct parts equals the number of partitions of weight $n$ into odd parts.  

A clear refinement along these lines is to require distinct parts based on parity; i.e., to require either all of the even parts to be distinct or all of the odd parts to be distinct (while allowing the frequency of the other parts to be unrestricted).  This leads to two types of partitions, those known as PED partitions (wherein the even parts must be distinct and the odd parts are unrestricted) and POD partitions (wherein the odd parts must be distinct and the even parts are unrestricted).  We then define two corresponding enumerating functions, $\ped(n)$ which counts the number of PED partitions of weight $n$, and $\mypod(n)$ which counts the number of POD partitions of weight $n.$ These two functions have been studied from a variety of perspectives; see, for example, \cite{And2009, AHS, BalMer, BalWelDM, Chen1, Chen2, CuiGu, CuiGuMa, Dai, FangXueYao, HS2010, HS_2014_JIS, KeithZan, Merca, RaduSel, Wang, Xia}.  

Recently, Ballantine and Welch \cite{BalWel} generalized and refined these two functions in numerous ways.  One of the outcomes of their work was to consider PEND partitions and POND partitions, wherein the even (respectively, odd) parts are {\bf not allowed} to be distinct.  In a vein similar to that shared above, we let $\pend(n)$ denote the number of PEND partitions of weight $n$, and $\pond(n)$ denote the number of POND partitions of weight $n.$  The first several values of $\pend(n)$ appear in the OEIS \cite[A265254]{OEIS}, while the first several values of $\pond(n)$ appear in \cite[A265256]{OEIS}.

In light of the work of Ballantine and Welch \cite{BalWel}, the first author \cite{SellJIS}
proved the following Ramanujan--like congruences satisfied by $\pond(n)$ and $\pend(n)$:

\begin{theorem}
\label{thm:pond_congs}
For all $n\geq 0,$
\begin{align*}
\pond(3n+2) \equiv\ & 0 \pmod{2},\\
\pond(27n+26) \equiv\ & 0 \pmod{3}, \textrm{\ \ and}\\
\pond(3n+1) \equiv\ & 0 \pmod{4}.
\end{align*}
\end{theorem}

\begin{theorem}
\label{thm:pend_cong}
For all $n\geq 0,$    
\begin{equation*}
\pend(27n+19) \equiv  0 \pmod{3}.    
\end{equation*}
\end{theorem}

In the same work \cite{SellJIS}, he went on to prove the following infinite families of non--nested Ramanujan--like congruences modulo 3 by induction.  

\begin{theorem}
\label{thm:pond_infinite_family_congs}
For all  $\alpha \geq 1$ and all $n\geq 0,$   
$$
\pond\left(3^{2\alpha +1}n+\frac{23\cdot 3^{2\alpha}+1}{8}\right) \equiv 0 \pmod{3}.
$$
\end{theorem}

\begin{theorem}
\label{thm:pend_infinite_family_congs}
For all  $\alpha \geq 1$ and all $n\geq 0,$   
$$
\pend\left(3^{2\alpha +1}n+\frac{17\cdot 3^{2\alpha}-1}{8}\right) \equiv 0 \pmod{3}.
$$
\end{theorem}

All of the proof techniques used to prove Theorems \ref{thm:pond_congs}--\ref{thm:pend_infinite_family_congs} \cite{SellJIS} are elementary, relying on classical $q$-series identities and generating function manipulations, along with mathematical induction.

In the aforementioned work, little--to--no connection is made between PED, POD, PEND, and POND partitions, other than their obvious relationship in terms of the definition of the objects in question.  In particular, no connection is drawn between the infinite families of congruences mentioned in Theorems  \ref{thm:pond_infinite_family_congs} and \ref{thm:pend_infinite_family_congs}.  In this work, the primary goal is to shine a light on an unexpected but clear relationship between the above families of congruences for the PEND and POND partitions by considering the corresponding generating functions as modular objects.  In the process, we utilize an Atkin--Lehner involution to make the connection even more explicit, thus showing that the existence of one of these families of congruences immediately implies the other.   


\section{Preliminaries}
\label{sec:prelims}

Throughout this work, we will use the following shorthand notation for $q$-Pochhammer symbols:
\begin{align*}
    f_r := (q^r;q^r)_\infty = (1-q^r)(1-q^{2r})(1-q^{3r})\dots
\end{align*}

For the work below, it will be very important to have the generating functions for the various partition functions in question.  We now take a moment to collect those results here.  We require one result multiple times below, so we explicitly mention it here.  
\begin{lemma}
\label{lemma:negq-negq}
We have 
$$
(-q;-q)_\infty =\frac{f_2^3}{f_1f_4}.
$$
\end{lemma}
\begin{proof}
The proof of this lemma involves the following elementary manipulations. We have
\begin{align*}
(-q;-q)_\infty 
&=
(q^2;q^2)_\infty(-q;q^2)_\infty \\
&=
f_2\cdot \frac{(q^2;q^4)_\infty}{(q;q^2)_\infty} \\
&=
f_2\cdot \frac{(q^2;q^2)_\infty}{(q^4;q^4)_\infty}\cdot \frac{(q^2;q^2)_\infty}{(q;q)_\infty}\\
&=
\frac{f_2^3}{f_1f_4}.
\end{align*}
\end{proof}

Next, as was proven in \cite{SellJIS}, we have the following generating functions for $\pend(n)$ and $\pond(n)$, respectively.  
\begin{theorem}
\label{thm:pend_genfn}
We have
\begin{align}
\sum_{n=0}^{\infty} \pend(n)q^n &=\frac{f_2f_{12}}{f_1f_4f_6},\label{pendgenA}\\
\sum_{n=0}^{\infty} \pond(n)q^n &=\frac{f_4f_6^2}{f_2^2f_3f_{12}}.\label{pondgenA}
\end{align}
\end{theorem}

Given the close relationship of $\pend(n)$ and $\pond(n)$, we might well suspect a natural way of mapping between these two generating functions.  One possible approach, analogous to the methods used by Garvan and Morrow \cite{GarvanM}, as well as Chern and Tang \cite{Chern}, is to connect the two generating functions via an Atkin--Lehner involution.  However, the inequality between the number of factors on the right-hand sides of (\ref{pendgenA}), (\ref{pondgenA}) suggests that a direct involution is not possible.

Instead, we note recent work between the authors and Garvan in \cite{Garvan0}.  We first consider replacing $q$ by $-q$ in Theorem \ref{thm:pend_genfn}.
\begin{theorem}
\label{thm:pend_genfn_negq}
We have
\begin{align}
\sum_{n=0}^{\infty} (-1)^n \pend(n)q^n &=\frac{f_1f_{12}}{f_2^2f_6},\label{pendgenB}\\
\sum_{n=0}^{\infty} (-1)^n \pond(n)q^n &=\frac{f_3f_4}{f_2^2f_6}.\label{pondgenB}
\end{align}
\end{theorem}
\begin{proof}
To prove (\ref{pendgenB}) via elementary generating function manipulations, we have 
\begin{align*}
\sum_{n=0}^{\infty}  \pend(n)(-q)^n 
&=
\sum_{n=0}^{\infty} (-1)^n \pend(n)q^n \\
&=
\frac{f_2f_{12}}{f_4f_6}\cdot \frac{1}{(-q;-q)_\infty} \\
&=
\frac{f_{12}}{f_4f_6}\cdot \frac{f_1f_4}{f_2^2} \\
&=
\frac{f_1f_{12}}{f_2^2f_6}.
\end{align*}  Similarly, to prove (\ref{pondgenB}), we have 
\begin{align*}
\sum_{n=0}^{\infty}  \pond(n)(-q)^n 
&=
\sum_{n=0}^{\infty} (-1)^n \pond(n)q^n \\
&=
\frac{f_4f_6^2}{f_2^2f_{12}}\cdot \frac{1}{(-q^3;-q^3)_\infty} \\
&=
\frac{f_4f_6^2}{f_2^2f_{12}}\cdot \frac{f_3f_{12}}{f_6^3} \\
&=
\frac{f_3f_4}{f_2^2f_6}.
\end{align*}
\end{proof}

Notice the close resemblance between (\ref{pendgenB}) and (\ref{pondgenB}).  It is much more reasonable to suspect a direct mapping between these two via an Atkin--Lehner involution.  We will now show that this is in fact possible.


\section{Relating the PEND and POND Congruences via an Atkin--Lehner Involution}  

For this section we presume a knowledge of the theory of modular forms, especially with respect to the classical $U_{\ell}$ operator for $\ell$ a prime (for our purposes, $\ell=3$).  In particular, we let $\mathcal{M}\left( \Gamma_0(N) \right)$ denote the space of all modular functions over the congruence subgroup $\Gamma_0(N)$.  See \cite{Knopp} for a treatment of these concepts, as well as \cite{Diamond} for a more detailed treatment.

For our purposes one especially important property of the $U_{\ell}$ operator is as follows \cite[Chapter 8]{Knopp}: for a given power series $f(q)$ in $q$, we have
\begin{align}
U_{\ell}\left( f\left( q^{\ell} \right) \right) = f(q).\label{ellreductionbyu}
\end{align}  This operator may be applied arbitrarily many times; indeed, we will define
\begin{align*}
U_{\alpha}(\cdot) :&= U_3^{2\alpha+1}(\cdot) = U_3\circ U_3\circ ... \circ U_3(\cdot),
\end{align*} in which $U_3$ is applied $2\alpha+1$ times in the latter expression.

With this composite operator, we now consider slight extensions of the generating functions of the $\pond$ and $\pend$ functions.  Define
\begin{align*}
P_{\alpha}^{(0)} :&= q^{\frac{ 3^{2\alpha} - 1}{8}}\cdot\frac{f_{3^{2\alpha+2}}^3}{f_{4\cdot 3^{2\alpha+1}}^2}\cdot\frac{f_4 f_6^2}{f_2^2 f_3 f_{12}},\\
P_{\alpha}^{(1)} :&= q^{\frac{ 7\cdot 3^{2\alpha} + 1}{8}}\cdot\frac{f_{3^{2\alpha+1}}^2 f_{3^{2\alpha+2}}^3 f_{4\cdot 3^{2\alpha+1}}^2}{f_{2\cdot 3^{2\alpha+1}}^6}\cdot\frac{f_2 f_{12}}{f_1 f_4 f_6}.
\end{align*}  We have chosen these functions very precisely, so that by applying $U_{\alpha}$ to $P_{\alpha}^{(\beta)}$ and taking advantage of (\ref{ellreductionbyu}), we will obtain modular functions over $\Gamma_0\left(12\right)$.  These resultant functions will enumerate the progressions of interest in Theorems \ref{thm:pond_infinite_family_congs}, \ref{thm:pend_infinite_family_congs}.

Indeed, if we take the standard substitution $q=\exp(2\pi i\tau)$ with $\tau\in\mathbb{H}$, we have
\begin{align*}
P_{\alpha}^{(0)} &= \frac{\eta(3^{2\alpha+2}\tau)^3}{\eta(4\cdot 3^{2\alpha+1}\tau)^2}\cdot\frac{\eta(4\tau) \eta(6\tau)^2}{\eta(2\tau)^2 \eta(3\tau) \eta(12\tau)},\\
P_{\alpha}^{(1)} &= \frac{\eta(3^{2\alpha+1}\tau)^2 \eta(3^{2\alpha+2}\tau)^3 \eta(4\cdot 3^{2\alpha+1}\tau)^2}{\eta(2\cdot 3^{2\alpha+1}\tau)^6}\cdot\frac{\eta(2\tau) \eta(12\tau)}{\eta(\tau) \eta(4\tau) \eta(6\tau)},
\end{align*} in which $\eta(\tau)$ is the Dedekind eta function \cite[Chapter 3]{Knopp}.  It can be proved by standard techniques (e.g., \cite[Theorem 1]{Newman}) that, for $\beta=0,1$, we have
\begin{align*}
P_{\alpha}^{(\beta)}\in\mathcal{M}\left( \Gamma_0\left(4\cdot 3^{2\alpha+2}\right) \right).
\end{align*}

We now take advantage of a useful theorem \cite[Lemma 7]{AtkinL}:
\begin{theorem}\label{reduceNwithu}
If $\ell$ is a prime, and $N\in\mathbb{Z}$ such that $\ell^2 | N$, then for any $f\in\mathcal{M}\left( \Gamma_0\left(N\right) \right)$ we have
\begin{align*}
U_{\ell}(f)\in\mathcal{M}\left( \Gamma_0\left(N/\ell\right) \right).
\end{align*}
\end{theorem}

If we denote
\begin{align*}
L_{\alpha}^{(0)} &= U_{\alpha}\left(P_{\alpha}^{(0)}\right),\\
L_{\alpha}^{(1)} &= U_{\alpha}\left(P_{\alpha}^{(1)}\right),
\end{align*} then by Theorem \ref{reduceNwithu} we have
\begin{align*}
L_{\alpha}^{(\beta)} = U_{\alpha}\left( P_{\alpha}^{(\beta)} \right)&\in\mathcal{M}\left( \Gamma_0\left(4\cdot 3^{2\alpha+2}/3^{2\alpha+1}\right)\right) =\mathcal{M}\left( \Gamma_0\left(12\right)\right).
\end{align*}

Let us first examine the application of $U_{\alpha}$ to $P_{\alpha}^{(\beta)}$.  We have
\begin{align*}
U_{\alpha}\left(P_{\alpha}^{(0)}\right) &= \frac{f_{3}^3}{f_{4}^2}\cdot U_{\alpha}\left(q^{\frac{ 3^{2\alpha} - 1}{8}}\cdot\frac{f_4 f_6^2}{f_2^2 f_3 f_{12}}\right)\\
&= \frac{f_{3}^3}{f_{4}^2}\cdot U_{\alpha}\left( \sum_{n\ge 0} \pond(n)q^{n+\frac{ 3^{2\alpha} - 1}{8}} \right)\\
&=\frac{f_{3}^3}{f_{4}^2}\cdot U_{\alpha}\left( \sum_{n\ge \frac{ 3^{2\alpha} - 1 }{8}} \pond\left(n - \frac{ \left(3^{2\alpha} - 1\right)}{8}\right)q^{n} \right)\\
&=\frac{f_{3}^3}{f_{4}^2}\cdot \sum_{3^{2\alpha+1}n\ge \frac{ 3^{2\alpha} - 1 }{8}} \pond\left(3^{2\alpha+1}n - \frac{ \left(3^{2\alpha} - 1\right)}{8}\right)q^{n}.
\end{align*}  Of course, rearranging the inequality restricting our sum, we have
\begin{align*}
3^{2\alpha+1}n&\ge \frac{ 3^{2\alpha} - 1 }{8},\\
3^{2\alpha}(24n-1)&\ge -1,\\
n&\ge 1.
\end{align*}  If we then shift our summand so that $n\ge 0$, then we derive the sequences of interest to us from Theorem \ref{thm:pond_infinite_family_congs}:
\begin{align*}
U_{\alpha}\left(P_{\alpha}^{(0)}\right) &= \frac{f_{3}^3}{f_{4}^2}\cdot \sum_{n\ge 1} \pond\left(3^{2\alpha+1}n - \frac{ \left(3^{2\alpha} - 1\right)}{8}\right)q^{n},\\
&= \frac{f_{3}^3}{f_{4}^2}\cdot \sum_{n\ge 0} \pond\left(3^{2\alpha+1}(n+1) - \frac{ \left(3^{2\alpha} - 1\right)}{8}\right)q^{n+1},\\
&= \frac{f_{3}^3}{f_{4}^2}\cdot \sum_{n\ge 0} \pond\left(3^{2\alpha+1}n + \frac{ \left(23\cdot 3^{2\alpha} + 1\right)}{8}\right)q^{n+1}.
\end{align*}  In similar manner, it can be shown that
\begin{align*}
U_{\alpha}\left(P_{\alpha}^{(1)}\right) &= \frac{f_1^2f_3^3f_4^2}{f_2^6}\cdot \sum_{n\ge 0} \pend\left(3^{2\alpha+1}n + \frac{ \left(17\cdot 3^{2\alpha} + 1\right)}{8}\right)q^{n+1}.
\end{align*}  Therefore, $L_{\alpha}^{(\beta)}$ exhibits the subsequences of interest to us.  As a consequence of Theorems \ref{thm:pond_infinite_family_congs}, \ref{thm:pend_infinite_family_congs}, we have already proved that
\begin{align*}
L_{\alpha}^{(\beta)}\equiv 0\pmod{3}.
\end{align*} for all $\alpha\ge 0$ and $\beta\in\{0,1\}$.

We suspected a close relationship between these two functions.  To approach that, let us first return to $P_{\alpha}^{(\beta)}$.
We define operator $\nu$, which acts on any power series in $q$ by sending $q$ to $-q$:
\begin{align*}
\nu: f(q)\rightarrow f(-q).
\end{align*}  Next, we consider certain useful operators which we call \textit{Atkin--Lehner involutions}.  These were introduced in \cite{AtkinL}, and are concerned primarily with the theory of cusp forms, but they are useful in the general theory of modular forms.  We define them as follows (e.g., \cite[Definition 2.19]{Ono0}):
\begin{definition}\label{ALdefn}
Let $N$ be a positive integer, $p$ a prime divisor of $N$, and $q = p^k$ the maximum power of $p$ which divides $N$.  The associated Atkin--Lehner involution for $\Gamma_0(N)$ is the operator induced by the matrix
\begin{align*}
W_{N,q}:=\begin{pmatrix}
qA & B\\
NC & qD
\end{pmatrix},
\end{align*} in which $A,B,C,D\in\mathbb{Z}$, and $\mathrm{det}(W)=q$.
\end{definition}
These operators are well-defined irrespective of the precise choice of $A,B,C,D$ \cite[Lemma 10]{AtkinL}, \cite[Remark 2.20]{Ono0}.  We will use the Atkin--Lehner involution for $\Gamma_0\left(4\cdot 3^{2\alpha+2}\right)$ and $q=4$, which we define with the following matrix:
\begin{align*}
W:= W_{4\cdot 3^{2\alpha+2},4} = \begin{pmatrix}
4 & -1\\
4\cdot 3^{2\alpha+2} & 1-3^{2\alpha+2}
\end{pmatrix}.
\end{align*}  We are interested in the conjugation mapping
\begin{align*}
\gamma := \nu W\nu.
\end{align*}  First we apply $\nu$ to $P_{\alpha}^{(0)}$.  Certainly the leading power of $q$ will induce an alternating sign depending on $\alpha$.  Moreover, we know that $\nu$ will not affect any factor of the form $f_k$ when $k$ is even; and we know how $f_k$ is to be transformed for odd $k$ by recalling Lemma \ref{lemma:negq-negq} above.  We therefore have 
\begin{align*}
P_{\alpha}^{(0)}|_{\nu} = (-1)^{\frac{ 3^{2\alpha} - 1}{8}} q^{\frac{ 3^{2\alpha} - 1}{8}}\cdot\frac{f_{2\cdot 3^{2\alpha+2}}^9}{f_{3^{2\alpha+2}}^3 f_{4\cdot 3^{2\alpha+2}}^3 f_{4\cdot 3^{2\alpha+1}}^2}\cdot\frac{f_3 f_4}{f_2^2 f_{6}}.
\end{align*}  Now we need to apply the involution associated with $W$.  In so doing, we utilize the transformation properties of the Dedekind eta function \cite[Chapters 3-4]{Knopp}: for any $\left(\begin{smallmatrix}
a & b\\
c & d
\end{smallmatrix}\right)\in\mathrm{SL}\left( 2,\mathbb{Z} \right)$, we have
\begin{align}
\eta\left( \frac{a\tau+b}{c\tau+d} \right) = \left(-i(c\tau+d)\right)^{1/2}\cdot\delta(a,b,c,d)\cdot\eta(\tau),\label{etatransform}
\end{align} in which we take the corresponding principal branch of the square root, and $\delta$ is a certain 24th root of unity depending on $a,b,c,d$.

Of course, $W$ is not a member of $\mathrm{SL}\left( 2,\mathbb{Z} \right)$, since it has determinant 4.  But we can account for this in a straightforward manner.  If $\left(\begin{smallmatrix}
A & B\\
C & D
\end{smallmatrix}\right)\in\mathrm{GL}\left( 2,\mathbb{Z} \right)$, and $g:=\mathrm{gcd}(A,C)$, we can express
\begin{align}
\begin{pmatrix}
A & B\\
C & D
\end{pmatrix} = \begin{pmatrix}
A/g & -y\\
C/g & x
\end{pmatrix}.\begin{pmatrix}
g & Bx+Dy\\
0 & (AD-BC)/g,
\end{pmatrix}\label{matrixexpansion}
\end{align} with $x,y\in\mathbb{Z}$ such that $Ax+Cy=g$.  Here $\left(\begin{smallmatrix}
A/g & -y\\
C/g & x
\end{smallmatrix}\right)\in\mathrm{SL}\left( 2,\mathbb{Z} \right)$.  For example, if we apply $W$ to $\eta(6\tau)$, we have the corresponding matrix product
\begin{align}
6W=\begin{pmatrix}
4\cdot 6 & -6\\
4\cdot 3^{2\alpha+2} & 1-3^{2\alpha+2}
\end{pmatrix} = \begin{pmatrix}
2 & 1\\
3^{2\alpha+1} & \frac{1}{2}\left( 1+3^{2\alpha+1} \right)
\end{pmatrix}.\begin{pmatrix}
12 & -4\\
0 & 2
\end{pmatrix}.\label{matrixexample}
\end{align}  Notice that by combining (\ref{etatransform}) with (\ref{matrixexample}), we see that the eta factor associated with $\eta(6W\tau)$ will be $\eta((12\tau-4)/2) = \eta(6\tau-2)$.  More generally, a straightforward computation from Definition \ref{ALdefn} shows that if we apply an Atkin--Lehner operator with $\Gamma_0(N)$ to $\eta(\delta\tau)$ with $\delta|N$, we will produce an eta factor of the form $\eta(\lambda\tau+\kappa)$ with $\lambda|N$ and $\kappa\in\mathbb{Z}$.  We see that these operators are in a certain sense ``well-behaved" when applied to eta quotients.

Thus we have a straightforward process of applying $W$ to $P_{\alpha}^{(0)}|_{\nu}$.  Doing so gives us the following:
\begin{align*}
&P_{\alpha}^{(0)}|_{W\circ\nu}\\
&= 2(-1)^{\frac{ 3^{2\alpha} - 1}{8}}\epsilon\cdot \frac{\eta\left( 2\cdot 3^{2\alpha+2}\tau-\frac{1}{2}\left( 1+3^{2\alpha+2} \right) \right)^9}{\eta\left( 4\cdot 3^{2\alpha+2}\tau - 3^{2\alpha+2} \right)^3 \eta\left( 3^{2\alpha+2}\tau - \frac{3}{4}\left( 1+3^{2\alpha+1} \right) \right)^3 \eta\left( 3^{2\alpha+1}\tau - \frac{1}{4}\left( 1+3^{2\alpha+1} \right) \right)^2}\cdot\frac{\eta\left( \tau \right)\eta\left( 12\tau - 3 \right)}{\eta\left( 2\tau \right)^2 \eta\left( 6\tau - 2 \right)}\\
&=2(-1)^{\frac{ 3^{2\alpha} - 1}{8}}\epsilon' \cdot \frac{\eta\left( 2\cdot 3^{2\alpha+2}\tau\right)^9}{\eta\left( 4\cdot 3^{2\alpha+2}\tau \right)^3 \eta\left( 3^{2\alpha+2}\tau \right)^3 \eta\left( 3^{2\alpha+1}\tau\right)^2}\cdot\frac{\eta\left( \tau \right)\eta\left( 12\tau\right)}{\eta\left( 2\tau \right)^2 \eta\left( 6\tau\right)}\\
&= 2(-1)^{\frac{ 3^{2\alpha} - 1}{8}}\epsilon' \cdot q^{\frac{ 7\cdot 3^{2\alpha} + 1}{8}}\cdot \frac{f_{2\cdot 3^{2\alpha+2}}^9}{f_{4\cdot 3^{2\alpha+2}}^3 f_{3^{2\alpha+2}}^3 f_{3^{2\alpha+1}}^2}\cdot\frac{f_1 f_{12}}{f_2^2 f_6}.
\end{align*}  The Nebentypus $\epsilon'$ reduces to
\begin{align*}
\epsilon' = (-1)^{\frac{ 7\cdot 3^{2\alpha} + 1}{8}}.
\end{align*}  Thus, we can now apply $\nu$ again, and we achieve the following:
\begin{align*}
P_{\alpha}^{(0)}|_{\nu\circ W\circ\nu} &= 2(-1)^{\frac{ 3^{2\alpha} - 1}{8}}\cdot q^{\frac{ 7\cdot 3^{2\alpha} + 1}{8}}\cdot \frac{f_{3^{2\alpha+1}}^2 f_{3^{2\alpha+2}}^3 f_{4\cdot 3^{2\alpha+1}}^2}{f_{2\cdot 3^{2\alpha+1}}^6}\cdot\frac{f_2 f_{12}}{f_1 f_4 f_6}\\
&= 2(-1)^{\frac{ 3^{2\alpha} - 1}{8}}\cdot P_{\alpha}^{(1)}.
\end{align*}  Finally, we note that
\begin{align*}
\frac{ 3^{2\alpha} - 1}{8} \equiv \alpha \pmod{2},
\end{align*} from which we thus obtain
\begin{align*}
P_{\alpha}^{(0)}|_{\gamma} = 2(-1)^{\alpha}\cdot P_{\alpha}^{(1)}.
\end{align*}  At first this appears only to apply to the generating functions.  However, we can take advantage of an extremely powerful theorem \cite{AtkinL}:
\begin{theorem}
The operators $U_{\ell}$ and $W$ commute, provided that $\mathrm{gcd}(\ell, \mathrm{det}(W))=1$.
\end{theorem}  Moreover, the operators $U_{\ell}$ and $\nu$ trivially commute when $\ell$ is odd, as in our case.  Therefore, we have
\begin{align}
L_{\alpha}^{(0)}|_{\gamma} &= P_{\alpha}^{(0)}|_{\gamma\circ U_{\alpha}}\nonumber\\
&= P_{\alpha}^{(0)}|_{U_{\alpha}\circ \gamma}\nonumber\\
&= \left( 2(-1)^{\alpha}\cdot P_{\alpha}^{(1)} \right)|_{U_{\alpha}}\nonumber\\
&= 2(-1)^{\alpha}\cdot \left( P_{\alpha}^{(1)} \right)|_{U_{\alpha}}\nonumber\\
&= 2(-1)^{\alpha}\cdot  L_{\alpha}^{(1)}.\label{gammala1}
\end{align}  Thus, up to an alternating sign and a factor of 2, there is a natural mapping between these two function sequences.

This factor of 2 is noteworthy.  Notice that by Theorem \ref{thm:pond_congs}, $\pond(n)$ is even when $n$ is not divisible by 3.  Since the progressions in Theorem \ref{thm:pond_infinite_family_congs} have a power of 3 as a base, and initial terms \textit{not} divisible by 3, we know that
\begin{align*}
L_{\alpha}^{(0)}\equiv 0\pmod{2}.
\end{align*}  We do \textit{not} have an analogous parity property for $L_{\alpha}^{(1)}$.  Thus any direct mapping between $L_{\alpha}^{(0)}$ and $L_{\alpha}^{(1)}$ must account for this discrepancy.

How does this mapping manifest itself directly between the two sequences?  To approach that, we first define 
\begin{align*}
x :&= q\frac{f_2^2 f_3 f_{12}^3}{f_1^3 f_4 f_6^2},\\
z_1 :&= \frac{f_1^3 f_{12}}{f_3 f_4^3},\\
z_2 :&= \frac{f_2^9 f_3 f_{12}^2}{f_1^3 f_4^6 f_6^3},\\
y :&= \frac{f_1^4 f_4^4 f_6^{10}}{f_2^{10}f_3^4 f_{12}^4}.
\end{align*}  By Ligozat's theorem \cite[Theorem 23]{Radu}, we can show that the first three are Hauptmoduln for $\mathrm{X}_0(12)$ at the cusp $[1/4]$, while $y$ is the same at the cusp $[1/2]$.  Moreover, these are subject to the relations
\begin{align*}
z_1 &= 1-3x,\\
z_2 &= 1+3x,\\
y &= \frac{1-x}{1+3x}.
\end{align*}  In particular, $z_1$ and $z_2$ have positive order at the cusps $[0]$, $[1/2]$, respectively.  We note that each $L_{\alpha}^{(\beta)}$ can be shown \cite[Theorem 39]{Radu} to have negative order at these same cusps, as well as $[1/4]$, but nonnegative order elsewhere.

This allows us to embed each $L_{\alpha}^{(\beta)}$ into the localization ring $\mathbb{Z}[x]_{\mathcal{S}(x)}$, in which $\mathcal{S}(x)$ is the multiplicatively closed set of all nonnegative powers of $z_1$ and $z_2$, i.e., the set of all nonnegative powers of $1\pm 3x$.

The usefulness of $y$ becomes clear when we we apply $\gamma$ to $x$, yielding
\begin{align*}
x|_{\gamma} = \frac{1-x}{1+3x} = y.
\end{align*}  For example, if we embed $L_1^{(0)}$ into $\mathbb{Z}[x]_{\mathcal{S}(x)}$, we have the representation
\begin{align*}
L_1^{(0)} = \frac{6}{(1-3x)^{11}(1+3x)^8} 
&(135 x + 7784 x^2 + 246674 x^3 + 2988256 x^4 + 29350555 x^5\\
 &+ 147829632 x^6 + 639848952 x^7 + 1430421984 x^8 + 2416633758 x^9\\
 &+ 754252560 x^{10} - 3857299380 x^{11} - 6302175840 x^{12} - 3695798178 x^{13}\\ 
 &+ 3775356864 x^{14} + 6406422840 x^{15} + 765275040 x^{16} - 1944188325 x^{17}\\
 & - 573956280 x^{18} + 9565938 x^{19} - 4782969 x^{21}).
\end{align*}  Notice that this representation exhibits divisibility by 6, accounting for the factor of 3 that we predict in these progressions.

Now, if we apply $\gamma$ to $L_1^{(0)}$, we can simply map $x\mapsto y$:
\begin{align*}
L_1^{(0)}|_{\gamma} = \frac{6}{(1-3y)^{11}(1+3y)^8} 
&(135 y + 7784 y^2 + 246674 y^3 + 2988256 y^4 + 29350555 y^5\\
 &+ 147829632 y^6 + 639848952 y^7 + 1430421984 y^8 + 2416633758 y^9\\
 &+ 754252560 y^{10} - 3857299380 y^{11} - 6302175840 y^{12} - 3695798178 y^{13}\\ 
 &+ 3775356864 y^{14} + 6406422840 y^{15} + 765275040 y^{16} - 1944188325 ^{17}\\
 & - 573956280 y^{18} + 9565938 y^{19} - 4782969 y^{21})\\
&= -2 L_1^{(1)},
\end{align*} whence
\begin{align*}
L_1^{(1)} = \frac{-3}{(1-3y)^{11}(1+3y)^8} 
&(135 y + 7784 y^2 + 246674 y^3 + 2988256 y^4 + 29350555 y^5\\
 &+ 147829632 y^6 + 639848952 y^7 + 1430421984 y^8 + 2416633758 y^9\\
 &+ 754252560 y^{10} - 3857299380 y^{11} - 6302175840 y^{12} - 3695798178 y^{13}\\ 
 &+ 3775356864 y^{14} + 6406422840 y^{15} + 765275040 y^{16} - 1944188325 ^{17}\\
 & - 573956280 y^{18} + 9565938 y^{19} - 4782969 y^{21}).
\end{align*}  We see that $L_1^{(1)}$ is enumerated by exactly the same rational polynomial as $L_1^{(0)}$.  There are only two differences: first the Hauptmodul $y$ is here substituted in place of $x$.  Secondly, we see that $L_1^{(1)}$ does not exhibit the same divisibility by 2 that is so shown by $L_1^{(0)}$.  Indeed, this factor is effectively removed in mapping from one to the other.

More generally, by (\ref{gammala1}) we have the following:
\begin{theorem}
Define the isomorphism
\begin{align*}
\sigma&: \mathbb{C}[x]_{\mathcal{S}(x)} \longrightarrow \mathbb{C}[y]_{\mathcal{S}(y)}\\
&: x\longmapsto y.
\end{align*} Then for all $\alpha\ge 1$, $\beta\in\{0,1\}$, we have
\begin{align*}
\sigma\left( L_{\alpha}^{(\beta)} \right) = (-1)^{\alpha} 2^{1-2\beta}L_{\alpha}^{(1-\beta)}.
\end{align*}
\end{theorem}


\section{Closing Thoughts}

\subsection{Multiplicities Hold Additional Arithmetical Information}

The most striking aspect of these results is the discovery that congruence multiplicities can contain precise arithmetical information---indeed, more even than the common divisibility properties uniting two different congruences.

The usefulness of this property becomes clear when we consider the possibility of proving standing conjectures.  Clearly one of the appeals to congruence multiplicity is that one can often prove a given congruence relation by appealing to work that has already been done, in lieu of a long and tedious direct proof.

Suppose we are studying a conjectured congruence with respect to a given prime---say, 3, as in our case.  To apply a proof by multiplicity, we would of course look at previously established congruence results modulo 3 for other partitions, in the hope of finding a connection with our conjectured result.  What is interesting here is that these previous results can involve more complex relationships---say, divisibility by 6, or 15, or 21---even if the newly conjectured result lacks such additional divisibility properties.

In this manner, one has a broader array of tools for resolving a given congruence conjecture.

\subsection{PED and POD Partitions}

As was mentiond in the introductory comments above, PED and POD partitions have also been studied from an arithmetic perspective.  In particular, we highlight the following previously proven families of congruences satisfied by $\ped(n)$ and $\mypod(n)$:

\begin{theorem}[Andrews, Hirschhorn, and Sellers \cite{AHS}]
\label{thm:ped_infinite_family_congs}    
For all  $\alpha \geq 1$ and all $n\geq 0,$   
$$
\ped\left(3^{2\alpha +1}n+\frac{17\cdot 3^{2\alpha}-1}{8}\right) \equiv 0 \pmod{3}.
$$
\end{theorem}

\begin{theorem}[Hirschhorn and Sellers \cite{HS2010}]
\label{thm:pod_infinite_family_congs}    
For all  $\alpha \geq 1$ and all $n\geq 0,$   
$$
\mypod\left(3^{2\alpha +1}n+\frac{23\cdot 3^{2\alpha}+1}{8}\right) \equiv 0 \pmod{3}.
$$
\end{theorem}
The similarity between Theorems \ref{thm:ped_infinite_family_congs} and \ref{thm:pod_infinite_family_congs} suggest a correspondence not unlike that exhibited by Theorems \ref{thm:pond_infinite_family_congs} and \ref{thm:pend_infinite_family_congs}.  (Consider, for example, the striking similarities between the arithmetic progressions involved.)  The interested reader may wish to further pursue these connections.

\section{Acknowledgments}

This research was funded in whole or in part by the Austrian Science Fund (FWF) Principal Investigator Project 10.55776/PAT6428623, ``Towards a Unified Theory of Partition Congruences."  For open access purposes, the authors have applied a CC BY public copyright license to any author-accepted manuscript version arising from this submission.

We would like to thank the Austrian Government and People for their generous support.  We are also grateful for the very helpful advice of the anonymous referee.

\end{document}